\documentclass[preprint,17pt, sort&compress]{elsarticle} 

\usepackage{lineno,hyperref}
\usepackage{upgreek}
\modulolinenumbers[5]
\usepackage[margin=2.5cm]{geometry}
\journal{Journal of \LaTeX\ Templates}
\usepackage{tikz}
\usepackage{circuitikz}

\usepackage{enumitem}

\newtheorem{definition }{\sc \bf Definition }
\usepackage{mathtools}
\newtheorem{definition}{Definition}

\usepackage{amsmath} %
\usepackage{graphicx}
\usepackage{float}
\usepackage{units}
\usepackage{bm}
\DeclareMathOperator{\sech}{sech}
\usepackage{graphicx}
\makeatletter
\def\fps@figure{hbtp}
\def\fps@table{hbtp}
\makeatother

\newtheorem{example}{\sc \bf Example}  
\newtheorem{theorem}{Theorem}
\newcommand{\proofofref}{}
\newproof{zproofof}{Proof of \proofofref}

\newtheorem{lemma}{\sc \bf Lemma}  
\usepackage{units}
\usepackage{bm}
\usepackage{amssymb,amsmath}
\usepackage{upgreek}
\usepackage[section]{placeins}
\usepackage{upgreek}
\usepackage{booktabs} 
\usepackage{multicol}
\usepackage{multirow}
\usepackage[utf8]{inputenc}
\usepackage{amssymb} %
\usepackage{latexsym} %

\usepackage{graphicx}
\hypersetup{colorlinks,linkcolor={blue},citecolor={blue},urlcolor={red}} 
\begin{document}
\begin{frontmatter}
\title{Numerical approximation of nonlinear stochastic Volterra integral equation using Walsh function }

\author[addresss1]{Prit Pritam Paikaray}
\ead{paikaraypritpritam@gmail.com}
\author[addresss1]{Nigam Chandra Parida}
\ead{ncparida@gmail.com}
\author[addresss1]{Sanghamitra Beuria}
\ead{sbeuria108@gmail.com}
\cortext[mycorrespondingauthor]{Corresponding author}
\author[address2]{Omid Nikan}
\ead{omidnikan77@yahoo.com}
\address[addresss1]{\normalfont Department of Mathematics, College of Basic Science and Humanities, OUAT, Bhubaneswar, 751003, India}
\address[address2]{\normalfont School of Mathematics and Computer Science, Iran University of Science and Technology, Narmak, Tehran, Iran}

\begin{abstract}
This paper adopts a highly effective numerical approach for  approximating non-linear stochastic Volterra integral equations (NLSVIEs) based on the operational matrices of the Walsh function and the collocation method. The method transforms the integral equation into a system of algebraic equations, which allows for the derivation of an approximate solution. Error analysis has been performed, confirming the effectiveness of
the proposed method, which results in a linear order of convergence. Numerical examples are provided to  illustrate the precision and effectiveness of this proposed method.
 \end{abstract}

 	\end{frontmatter}
 	\section{Introduction} 
Non-linear stochastic Volterra integral equations (NLSVIEs) have found various applications in the biological sciences, particularly in modeling and simulating biological systems.  NLSVIEs have been used to model population growth and extinction in different species. For instance, they have been used to study the dynamics of predator-prey systems, where the population of predators and prey interact with each other in a non-linear way. NLSVIEs have been used to study the spread of infectious diseases in populations. In this case, the equations model the dynamics of the disease transmission, including the rate of infection and recovery, and the effect of different control measures.
NLSVIEs have been used to model the dynamics of neuronal networks and synaptic plasticity in the brain. In this case, the equations model the non-linear interactions between neurons and how they change over time \cite{Kloeden,Oksendal}. However, it is impossible to obtain exact solutions for all NLSVIEs, various numerical techniques are adopted to obtain approximate solutions. In particular,  numerical techniques including orthogonal functions are often  applied to solve these problems.
 Haar wavelet \cite{Jiang}, Orthogonal functions such as block pulse function (BPF) \cite{Asgari}, and Legendre polynomials \cite{Zeghdane} have been implemented to simulate the solution of NLSVIEs. The Walsh function has also been employed for solving the stochastic Volterra integral equation \cite{Paikaray}.

In this work,  we investigate the approximated solution  $x(t)$ of the NLSVIE by using of the  the Walsh function \cite{Walsh} as
 \begin{equation}\label{NLSIE}
	x(t)=x_0+\int_{0}^{t}k_1(s,t)\beta(x(s)){\rm d}s+\int_{0}^{t}k_2(s,t)\sigma(x(s)){\rm d}B(s)
\end{equation}
where $x(t)$ s an unknown function to be determined and $k_1(s,t)$, $k_2(s,t)$ for $s,t\in[0,T)$, represent the stochastic processes based on the same probability space $(\Omega,F,P)$. Here $B(t)$ and $\int_{0}^{t}k_2(s,t)x(s){\rm d}B(s)$ denote the Brownian motion  and  It$\hat{o}$ integral, respectively \cite{Kloeden,Oksendal}.

The rest of current research paper is organized as follows.  
Section \ref{Walsh} introduces  the  Walsh Function and its properties.
Section \ref{sect3} describes Relationship between Walsh Function and Block Pulse Functions (BPFs).
Section \ref{sect3a} presents a  numerical technique by based on the operational matrices of the Walsh function and the collocation method to discretize the NLSVIEs. 
Section \ref{sect4} discusses the convergence and error analysis of the method to demonstrate the method's validity and precision.
Section \ref{sect5} gives  two numerical  examples  by using  the proposed method  to demonstrate the efficacy of the method.
 Finally, Section \ref{sect6}  contains some concluding remarks and summarizes the main findings. 
\section{Walsh Function and its Properties}\label{Walsh}
\begin{definition}\normalfont
	Rademacher function $r_i(t)$, $i=1,2,\hdots$, for $t\in[0,1)$ is defined  by \cite{Walsh} 
	
	$$r_i(t)=
	\Biggl\{
	\begin{aligned}
		1 & \quad \text{ $i=0$},\\
		{\rm sgn}(\sin(2^i\pi t)) & \quad \text{otherwise}\\
	\end{aligned}$$
	where,
	$${\rm sgn}(x)=
	\Biggl\{
	\begin{aligned}
		1 & \quad \text{ $x>0$},\\
		0 & \quad \text{ $x=0$},\\
		-1 & \quad \text{$x<0$}.
	\end{aligned}$$
\end{definition}
\begin{definition}\normalfont 
	\normalfont The $n^{th}$ Walsh function for $n=0,1,2,\cdots,$ denoted by $w_n(t)$, $t\in[0,1)$ is defined \cite{Walsh} as 
	$$w_n(t)=(r_q(t))^{b_q}.(r_{q-1}(t))^{b_{q-1}}.(r_{q-2}(t))^{b_{q-2}}\hdots (r_{1}(t))^{b_1}$$ where $n=b_q2^{q-1}+b_{q-1}2^{q-2}+b_{q-2}2^{q-3}+\hdots +b_12^{0}$ is the binary expression of $n$. Therefore, $q$, the number of digits present in the binary expression of $n$ is calculated by $q=\big[\log_2n\big]+1$ in which $\big[\cdot\big]$ is the greatest integer less than or equal's to $'\cdot'$.
\end{definition}

The first $m$ Walsh functions for $m \in \mathbb{N}$ can be written as an $m$-vector by
$W(t)=\begin{bmatrix}
	w_0(t) & w_1(t) &w_2(t)\hdots w_{m-1}(t)
\end{bmatrix}^T$, $t\in[0,1)$. The Walsh functions satisfy the following properties.

\subsection*{Orthonormality}
The set of Walsh functions is orthonormal. i.e., 
$$\int_{0}^{1}w_i(t)w_j(t){\rm d}t=\Biggl\{\begin{aligned}
	1 & \quad {i=j,}\\
	0 & \quad \text{otherwise}.
\end{aligned}$$
\subsection*{Completeness}
\noindent For every $f\in L^2([0,1))$ 
\begin{equation*}
	\int_{0}^{1}f^2(t){\rm d}t=\sum\limits_{i=0}^{\infty}f_i^2\lvert\lvert w_i(t)\lvert\lvert^2
\end{equation*}
where $f_i=\int_{0}^{1}f(t)w_i(t){\rm d}t$.
\subsection*{Walsh Function Approximation}
Any real-valued function $f(t)\in L^2([0,1))$ can be approximated as 
$$f_m(t)\simeq \sum\limits_{i=0}^{m-1}c_iw_i(t)$$
where, $c_i=\int_{0}^{1}f(t)w_i(t){\rm d}t$.\\
The matrix form can be represented by
\begin{equation}
	f(t)\simeq F^TT_WW(t) \label{Eq:2}
\end{equation}where
$ F=
\begin{bmatrix}
	f_0 &f_1 &f_2 \hdots f_{m-1}
\end{bmatrix}^T$, 
$f_i=\int_{ih}^{(i+1)h}f(s){\rm d}s$.\\
Here, $T_W=[w_i(\eta_j)]$ is called as the Walsh operational matrix where $\eta_j\in [jh, (j+1)h)$.\\
Similarly, function $k(s,t)\in L^2([0,1)\times[0,1))$ can be approximated by
$$k_m(s,t)\simeq \sum\limits_{i=0}^{m-1}\sum\limits_{j=0}^{m-1}c_{ij}w_i(s)w_j(t)$$
where $c_{ij}=\int_{0}^{1}\int_{0}^{1}k(s,t)w_i(s)w_j(t){\rm d}t{\rm d}s$
with the matrix form represented by
\begin{equation}
	k(s,t)\simeq W^T(s)T_WKT_WW(t)=W^T(t)T_WK^TT_WW(s) \label{Eq:3}
\end{equation}
where $K=[k_{ij}]_{m\times m}, k_{ij}=\int_{ih}^{(i+1)h}\int_{jh}^{(j+1)h}k(s,t){\rm d}t{\rm d}s$.
\section{Relationship between Walsh Function and Block Pulse Functions (BPFs)} \label{sect3}
\begin{definition}\normalfont [Block Pulse Functions]
	For a fixed positive integer $m$, an $m$-set of BPFs $\phi_i(t), t\in [0,1)$ for $i=0, 1,..., m-1$ is defined as
	$$\phi_i(t)=\biggl\{
	\begin{aligned}
		1 & \quad \text{if $\frac{i}{m}\le t < \frac{(i+1)}{m}, \quad$}\\
		0 & \quad \text{ otherwise}
	\end{aligned}
	$$
in which $\phi_i$ is known as the $i$th BPF.
\end{definition}
The set of all $m$ BPFs can be written concisely as an $m$-vector,
$\Phi(t)=\begin{bmatrix}
	\phi_0(t) & \phi_1(t) &\phi_2(t)\hdots \phi_{m-1}(t)
\end{bmatrix}^T$, $t\in[0,1)$.
The BPFs are disjoint, complete, and orthogonal \cite{Ganti}.
The BPFs in vector form satisfy 
$$ \Phi(t)\Phi(t)^TX=\tilde{X}\Phi(t) \;\textrm{and}\; \Phi^T(t)A\Phi(t)=\hat{A}\Phi(t)$$
where $X \in \mathbb{R}^{m \times 1}, \tilde{X}$ 
is the $m\times m$ diagonal matrix with $\tilde{X}(i, i)=X(i) \,\textrm{for}\, i=1, 2, 3\cdots m, A\in \mathbb{R}^{m \times m}$ and  $\hat{A}=\begin{bmatrix}
	a_{11}&	a_{22}&\hdots 	&a_{mm}
\end{bmatrix}^T$
is the $m$-vector with elements equal to the diagonal entries of $A$.
The integration of BPF vector $\Phi(t)$, $t\in[0,1)$ can be performed by \cite{Hatamzadeh}
\begin{equation}
	\int_{0}^{t}\Phi(\tau){\rm d}\tau=P\Phi(t), t\in[0,1), 
\end{equation}
Consequently, the integral of every function $f(t)\in L^2[0,1)$ can be estimated as
$$\int_{0}^{t}f(s){\rm d}s=F^TP\Phi(t).$$
The integration of the BPF vector $\Phi(t)$, with $t\in[0,1)$, via the It$\hat{o}$ integral can be executed by \cite{Maleknejad}
	\begin{equation}
		\int_{0}^{t}\Phi(\tau){\rm d}B(\tau)=P_S\Phi(t), t\in[0,1)
	\end{equation}
Hence, the It$\hat{o}$ integral of every function $f(t)\in L^2[0,1)$ can be represented as $$\int_{0}^{t}f(s){\rm d}B(s)=F^TP_S\Phi(t).$$
The Next theorem elucidates a correlation between the Walsh function and the block pulse function.
	\begin{theorem}\normalfont\cite{Paikaray}
		Let the $m$-set of Walsh function and BPF vectors be $W(t)$ and $\Phi(t)$ respectively. Then the BPF vectors $\Phi(t)$ can be used to approximate $W(t)$ as $W(t)=T_W\Phi(t)$, $m=2^k$, and $k=0,1,\hdots $, where $T_W=\big[c_{ij}\big]_{m\times m}$, $c_{ij}=w_i(\eta_j)$, for some  $\eta_j=\big(\frac{j}{m},\frac{j+1}{m}\big)$ and $i,j=0,1,2,\hdots m-1$.
	\end{theorem}	
	One can see that \cite{Cheng}  $$T_WT_W^T=mI \, \textrm{and} \, T_W^T=T_W$$ which implies that $\Phi(t)=\frac{1}{m}T_WW(t).$\\
	\begin{lemma}\normalfont\cite{Paikaray}[Integration of Walsh function]
		Suppose that $W(t)$ is a Walsh function vector, then the integral of $W(t)$ w.r.t. $t$ is given by \\
		$\int_{0}^{t}W(s){\rm d}s=\wedge W(t)$, where $\wedge =\frac{1}{m}T_WPT_W$ and $$P=\frac{1}{h}\begin{bmatrix}
			1 &2 &2&\hdots &2\\0 &1 &2&\hdots &2\\\vdots &\vdots &\vdots &\ddots &\vdots\\0 &0 &0&\hdots &1
		\end{bmatrix}$$
	\end{lemma}

	\begin{lemma}\normalfont\cite{Paikaray}[Stochastic integration of Walsh function]
		Suppose that $W(t)$ is a Walsh function vector, then the It$\hat{o}$ integral of $W(t)$ is given by\\
		$\int_{0}^{t}W(s){\rm d}B(s)=\wedge _S W(t)$, where $\wedge_S =\frac{1}{m}T_WP_ST_W$ and $$P_S=
		\begin{bmatrix}
			B(\frac{h}{2}) &B(h)&\hdots &B(h)\\0 &B(\frac{3h}{2})-B(h)&\hdots &B(2h)-B(h)\\ \vdots &\vdots &\ddots &\vdots \\0 &0&\hdots &B(\frac{(2m-1)h}{2})-B((m-1)h)
		\end{bmatrix}.$$
	\end{lemma}

 \section{Numerical method for NLSVIE} \label{sect3a}
Let us consider the NLSVIE  as
 	\begin{equation}\label{NLSIE}
 	x(t)=x_0+\int_{0}^{t}k_1(s,t)\beta(x(s)){\rm d}s+\int_{0}^{t}k_2(s,t)\sigma(x(s)){\rm d}B(s)
 \end{equation}
 where $x(t)$, $k_1(s,t)$, $k_2(s,t)$ for $s,t\in[0,T)$, represent the stochastic processes based on the same probability space $(\Omega,F,P)$ and $x(t)$ is unknown. Here $B(t)$ is Brownian motion \cite{Kloeden,Oksendal} and $\int_{0}^{t}k_2(s,t)x(s){\rm d}B(s)$ is the It$\hat{o}$ integral.\\
 Let $z_1(t)=\beta(x(s))$ and $z_2(t)=\sigma(x(s))$ which implies,
 $$x(t)=x_0+\int_{0}^{t}k_1(s,t)z_1(t){\rm d}s+\int_{0}^{t}k_2(s,t)z_2(t){\rm d}B(s).$$
 We can approximate $z_1(t)$, $z_2(t)$, $k_1(s,t)$, $k_2(s,t)$ for $s,t\in[0,T)$ as
  \begin{equation}
 	z_\eta(t)\simeq Z_\eta^TT_WW(t), \qquad  \eta=1,2,\label{Eq:zapp}
 \end{equation} 
 where  $ Z_\eta=
 \begin{bmatrix}
 	c_0 &c_1 &c_2 \hdots c_{m-1}
 \end{bmatrix}^T$
 and 
 $c_i=\int_{ih}^{(i+1)h}z_\eta(s){\rm d}s$.\\
 Similarly, for $\gamma=1,2$
 \begin{equation}
 	k_\gamma(s,t)\simeq W^T(s)T_WK_\gamma T_WW(t)=W^T(t)T_WK^T_\gamma T_WW(s) \label{Eq:Kapp}
 \end{equation}
 where $K_\gamma=[k_{ij}]_{m\times m}, k_{ij}=\int_{ih}^{(i+1)h}\int_{jh}^{(j+1)h}k_\gamma(s,t){\rm d}t{\rm d}s$.\\
 Assume that \begin{equation}
 	x(t)\simeq X^TT_WW(t), \label{Eq:xapp}
 \end{equation} 
 where  $ X=
 \begin{bmatrix}
 	x_0 &x_1 &x_2 \hdots x_{m-1}
 \end{bmatrix}^T$
 and 
 $x_i=\int_{ih}^{(i+1)h}x(s){\rm d}s$.
 Replacing \eqref{Eq:zapp}, \eqref{Eq:xapp} and \eqref{Eq:Kapp} in \eqref{NLSIE} leads to
 \begin{eqnarray}
 	X^TT_WW(t)&\simeq &x_0+\int_{0}^{t}W^T(t)T_WK^T_1T_WW(s)W^T(s)T_WZ_1{\rm d}s\nonumber\\
 	&&+\int_{0}^{t}W^T(t)T_WK^T_2T_WW(s)W^T(s)T_WZ_2 {\rm d}B(s)\nonumber\\
 	&=&x_0+W^T(t)T_WK^T_1T_W\int_{0}^{t}W(s)W^T(s)T_WZ_1{\rm d}s\nonumber\\
 	&&+W^T(t)T_WK^T_2T_W\int_{0}^{t}W(s)W^T(s)T_WZ_2 {\rm d}B(s)\label{laxman}
 \end{eqnarray}
 Now
 \begin{equation}
 	\int_{0}^{t}W(s)W^T(s)T_WZ_1{\rm d}s=T_W\tilde{Z_1}PT_WW(t)\label{int}
 \end{equation}
 Similarly,
 \begin{equation}
 	\int_{0}^{t}W(s)W^T(s)T_WZ_2{\rm d}B(s)
 	=T_W\tilde{Z_2}P_ST_WW(t)\label{intS}
 \end{equation}
 Inserting \eqref{int} and \eqref{intS} in \eqref{laxman} we obtain,
 \begin{eqnarray*}
 	X^TT_WW(t)&=&x_0+mW^T(t)T_WK^T_1\tilde{Z_1}PT_WW(t)\\
 	&&	+mW^T(t)T_WK^T_2\tilde{Z_2}P_ST_WW(t)\\
 	&=&x_0+W^T(t)T_WH_1T_WW(t)+W^T(t)T_WH_2T_WW(t)\\
 	&=&x_0+m\hat{H_1}^TT_WW(t)+m\hat{H_2}^TT_WW(t)
 \end{eqnarray*}
 i.e.,
 \begin{equation}
 	\Big(X^T-m\hat{H_1}^T-m\hat{H_2}^T\Big)T_WW(t)=x_0 \label{Eq:Final1}
 \end{equation}
and,
\begin{eqnarray}\label{Z}
	Z_1^TT_WW(t)=\beta(x_0+m\hat{H_1}^TT_WW(t)+m\hat{H_2}^TT_WW(t))\\\nonumber
	Z_2^TT_WW(t)=\sigma(x_0+m\hat{H_1}^TT_WW(t)+m\hat{H_2}^TT_WW(t))
\end{eqnarray}
in which $H_1=mK^T_1\tilde{Z_1}P$, $H_2=mK^T_2\tilde{Z_2}P_S$ and $\hat{H}_i$ denotes the $m$-vector with elements equal to the diagonal entries of $H_i$.\\
To calculate $Z_1$ and $Z_2$, we collocate the aforementioned equation \eqref{Z} at $t_j=\frac{2j+1}{2m}$ for $j=0,1,\hdots, m-1$ and solve the following system
 \begin{eqnarray}\label{Zf}
 	Z_1^TT_WW(t_j)=\beta(x_0+m\hat{H_1}^TT_WW(t_j)+m\hat{H_2}^TT_WW(t_j))\\\nonumber
 	Z_2^TT_WW(t_j)=\sigma(x_0+m\hat{H_1}^TT_WW(t_j)+m\hat{H_2}^TT_WW(t_j))
 \end{eqnarray}	
\section{Error Analysis} \label{sect4} 
\noindent This section focuses on analyzing the discrepancy between the approximate  and  exact solutions of the stochastic Volterra integral equation. Prior to commencing the analysis, we define the notation $E(|X|^2)^\frac{1}{2} = \|X\|_2$.\\
	\begin{theorem}\normalfont\label{fin}\cite{Paikaray}
If $f\in L^2[0,1)$ and fulfills the Lipschitz condition with a Lipschitz constant $C$, then the $2$-norm of $e_m(t)$ is $\mathcal{O}(h)$, where $e_m(t)=|f(t)-\sum\limits_{i=0}^{m-1}c_iw_i(t)|$ and $c_i=\int_{0}^{1}f(s)w_i(s){\rm d}s$.
\end{theorem}

\begin{theorem}\normalfont \label{Thk}\cite{Paikaray}
Assume that $k\in L^2\big([0,1)\times [0,1)\big)$ fulfills the Lipschitz condition with Lipschitz constant $L$. If  $k_m(x,y)=\sum\limits_{i=0}^{m-1}\sum\limits_{j=0}^{m-1}c_{ij}w_i(x)w_j(y)$, $c_{ij}=\int_{0}^{1}\int_{0}^{1}k(s,t)w_i(s)w_j(t){\rm d}t{\rm d}s$, then $\|e_m(x,y)\|_2=O(h)$, where $|e_m(x,y)|=|k(x,y)-k_m(x,y)|$.
\end{theorem}

\begin{theorem}\normalfont
Assume that $x_m(t)$ be the approximated solution of the NLSIE \eqref{NLSIE}. If \begin{enumerate}
\item[a.] $f\in L^2[0,1)$, $k_1(s,t)$ and $k_2(s,t)\in L^2\big( [0,1)\times[0,1)\big)$ fulfills the Lipschitz condition with Lipschitz constants with Lipschitz constants  $C$, $L_1$ and $L_2$ respectively
\item[b.] $|k_1(s,t)|\le\rho_1$, $|k_2(s,t)|\le\rho_2$, $|\beta(x(t))|\le \zeta_1$ and $|\sigma(x(t))|\le \zeta_2$, and
\item[c.] for $\eta_1,\eta_2\ge0$, $|\beta(x(t))-\beta(x_m(t))|\le \eta_1|x(t)-x_m(t)|$, $|\sigma(x(t))-\sigma(x_m(t))|\le \eta_2|x(t)-x_m(t)|$.
		\end{enumerate}
	Then,  we have $\|x(t)-x_m(t)\|_2^2=O(h^2)$.
	\end{theorem}
	\begin{proof}
	Consider the Volterra integral equation \eqref{NLSIE} and let $x_m(t)$ be the approximation of the solution obtained using the Walsh function. Then, we have
		\begin{eqnarray*}
			x(t)-x_m(t)&=&f(t)-f_m(t)\\
			&+&\int_{0}^{t}\big(k_1(s,t)\beta(x(s))-k_{1m}(s,t)\beta(x_m(s))\big){\rm d}s\\
			&+&\int_{0}^{t}\big(k_2(s,t)\sigma(x(s))-k_{2m}(s,t)\sigma(x_m(s))\big){\rm d}B(s)
		\end{eqnarray*}
		or,
		\begin{eqnarray*}
			|x(t)-x_m(t)|&\le&| f(t)-f_m(t)|\\\nonumber
			&+&\biggl|\int_{0}^{t}\big(k_1(s,t)\beta(x(s))-k_{1m}(s,t)\beta(x_m(s))\big){\rm d}s\biggr|\\\nonumber
			&+&\biggl|\int_{0}^{t}\big(k_2(s,t)\sigma(x(s))-k_{2m}(s,t)\sigma(x_m(s))\big){\rm d}B(s)\biggr|.\nonumber
		\end{eqnarray*}
		We know that, $(a+b+c)^2\le 3a^2+3b^2+3c^2$
		\begin{eqnarray*}
			|x(t)-x_m(t)|^2&\le&3| f(t)-f_m(t)|^2\\\nonumber
			&+&3\biggl|\int_{0}^{t}\big(k_1(s,t)\beta(x(s))-k_{1m}(s,t)\beta(x_m(s))\big){\rm d}s\biggr|^2\\\nonumber
			&+&3\biggl|\int_{0}^{t}\big(k_2(s,t)\sigma(x(s))-k_{2m}(s,t)\sigma(x_m(s))\big){\rm d}B(s)\biggr|^2.\nonumber
		\end{eqnarray*}
		which implies that
		\begin{eqnarray}\label{In}
			E\big(|x(t)-x_m(t)|^2\big)&\le&E\biggl(3| f(t)-f_m(t)|^2\biggr)\\\nonumber
			&+&E\biggl(3\biggl|\int_{0}^{t}\big(k_1(s,t)\beta(x(s))-k_{1m}(s,t)\beta(x_m(s))\big){\rm d}s\biggr|^2\biggr)\\\nonumber
			&+&E\biggl(3\biggl|\int_{0}^{t}\big(k_2(s,t)\sigma(x(s))-k_{2m}(s,t)\sigma(x_m(s))\big){\rm d}B(s)\biggr|^2\biggr).
		\end{eqnarray}
	Suppose, $$I_1=E\biggl(3\biggl|\int_{0}^{t}\big(k_1(s,t)\beta(x(s))-k_{1m}(s,t)\beta(x_m(s))\big){\rm d}s\biggr|^2\biggr)$$ and
	$$I_2=E\biggl(3\biggl|\int_{0}^{t}\big(k_2(s,t)\sigma(x(s))-k_{2m}(s,t)\sigma(x_m(s))\big){\rm d}B(s)\biggr|^2\biggr).$$
Applying the Theorem \ref{fin} in  inequality \eqref{In} results 
	\begin{equation}\label{sum}
		E\big(|x(t)-x_m(t)|^2\big)\le C^2h^2+I_1+I_2
	\end{equation}
		Now,
		\begin{eqnarray*}
			|k_1(s,t)\beta(x(s))-k_{1m}(s,t)\beta(x_m(s))|
			=&&|k_1(s,t)\beta(x(s))-k_1(s,t)\beta(x_m(s))\\
			&+&k_1(s,t)\beta(x_m(s))-k_{1m}(s,t)\beta(x_m(s))\|\\
			\le&&|k_1(s,t)||\beta(x(s))-\beta(x_m(s))|\\
			&+&|k_1(s,t)-k_{1m}(s,t)||\beta(x_m(s))|\\
			=&&|k_1(s,t)||\beta(x(s))-\beta(x_m(s))|\\
			&+&|k_1(s,t)-\beta k_{1m}(s,t)||\beta(x(s))-\beta(x(s))+\beta(x_m(s))|\\
			\le&&|k_1(s,t)||\beta(x(s))-\beta(x_m(s))|\\
			&+&|k_1(s,t)-k_{1m}(s,t)||\beta(x(s))|\\
			&+&|k_1(s,t)-k_{1m}(s,t)||\beta(x(s))-\beta(x_m(s))|\\ 	
		\end{eqnarray*}
		Let $|k_1(s,t)|\le\rho_1$, $|\beta(x(t))|\le \zeta_1$, $|\beta(x(t))-\beta(x_m(t))|\le \eta_1|x(t)-x_m(t)|$ and using Theorem \ref{Thk}, we get
		\begin{eqnarray*}
			|k_1(s,t)\beta(x(s))-k_{1m}(s,t)\beta(x_m(s))|&\le& \rho_1\eta_1|x(s)-x_m(s)|\\\nonumber
			&+&\sqrt{2}L_1h\zeta_1
			+\sqrt{2}L_1h\eta_1|x(s)-x_m(s)|
		\end{eqnarray*}
	which arrives at
	\begin{eqnarray}\label{normkernel}
		|k_1(s,t)\beta(x(s))-k_{1m}(s,t)\beta(x_m(s))|\le
		 \sqrt{2}L_1h\zeta_1+(\rho_1\eta_1+\sqrt{2}L_1h\eta_1)|x(s)-x_m(s)|
	\end{eqnarray}
		which gives,
		\begin{eqnarray*}
			I_1
			&\le&E\biggl(3\biggl(\int_{0}^{t}\biggl|k_1(s,t)\beta(x(s))-k_{1m}(s,t)\beta(x_m(s))\biggr|{\rm d}s\biggr)^2\biggr)\\
			&\le&E\biggl(3\biggl(\int_{0}^{t}\big(\sqrt{2}L_1h\zeta_1\\
			&+&(\rho_1\eta_1+\sqrt{2}L_1h\eta_1)|x(s)-x_m(s)|\big){\rm d}s\biggr)^2\biggr) \end{eqnarray*}
		In virtue of  the Cauchy- Schwarz inequality, for $t>0$ and $f\in L^2[0,t)$
		$$\biggl|\int_{0}^{t}f(s){\rm d}s\biggr|^2\le t\int_{0}^{t}|f|^2{\rm d}s$$
		this implies,
		\begin{eqnarray*}
		I_1
			&\le&E\biggl(3\int_{0}^{t}\biggl(\sqrt{2}L_1h\zeta\\
			&+&(\rho_1\eta_1+\sqrt{2}L_1h\eta_1)|x(s)-x_m(s)|\biggr)^2{\rm d}s\biggr)\\
			&\le&E\biggl(6\int_{0}^{t}\biggl(\bigl(\sqrt{2}L_1h\zeta_1\bigr)^2\\
			&+&\bigl((\rho_1\eta_1+\sqrt{2}L_1h\eta_1)|x(s)-x_m(s)|\bigr)^2\biggr){\rm d}s\biggr)\\
			&\le&E\biggl(6\bigl(\sqrt{2}L_1h\zeta_1\bigr)^2\\
			&+&6(\rho_1\eta_1+\sqrt{2}L_1h\eta_1)^2\int_{0}^{t}|x(s)-x_m(s)|^2{\rm d}s\biggr)
		\end{eqnarray*}
	Hence, 
	\begin{eqnarray}\label{eq:ink}
		I_2&\le& 6\bigl(\sqrt{2}L_1h\zeta_1\bigr)^2\\\nonumber
		&+&6(\rho_1\eta_1+\sqrt{2}L_1h\eta_1)^2\int_{0}^{t}E\bigl(|x(s)-x_m(s)|^2\bigr){\rm d}s.
	\end{eqnarray}
		Now,
		\begin{eqnarray*}
		I_2	\le&&	E\biggl(3\int_{0}^{t}\biggl|k_2(s,t)\sigma(x(s))-k_{2m}(s,t)\sigma(x_m(s))\biggr|^2{\rm d}s\biggr)
		\end{eqnarray*}
				Let $|k_2(s,t)|\le\rho_2$, $|\sigma(x(t))|\le \zeta_2$, $|\sigma(x(t))-\sigma(x_m(t))|\le \eta_2|x(t)-x_m(t)|$ and using Theorem \ref{Thk}, we get
				$$	|k_2(s,t)\sigma(x(s))-k_{2m}(s,t)\sigma(x_m(s))|\le
					\sqrt{2}L_2h\zeta_2+(\rho_2\eta_2+\sqrt{2}L_2h\eta_2)|x(s)-x_m(s)|$$
		\begin{eqnarray*}
			I_2\le&&
			E\biggl(3\int_{0}^{t}\big(\sqrt{2}L_2h\zeta_2+(\rho_2\eta_2+\sqrt{2}L_2h\eta_2)|x(s)-x_m(s)|\big)^2{\rm d}s\biggr)\\
			\le&&
			E\biggl(6\int_{0}^{t}\big((\sqrt{2}L_2h\zeta_2)^2+(\rho_2\eta_2+\sqrt{2}L_2h\eta_2)^2|x(s)-x_m(s)|^2\big){\rm d}s\biggr)\\
			&=&E\biggl(6(\sqrt{2}L_2h\zeta_2)^2+6(\rho_2\eta_2+\sqrt{2}L_2h\eta_2)^2\int_{0}^{t}|x(s)-x_m(s)|^2{\rm d}s\biggr)			
		\end{eqnarray*}
	Hence,
	\begin{eqnarray}\label{eq:inkb}
		I_2&\le& 6(\sqrt{2}L_2h\zeta_2)^2\\\nonumber
		&+&6(\rho_2\eta_2+\sqrt{2}L_2h\eta_2)^2\int_{0}^{t}E\bigl(|x(s)-x_m(s)|^2\bigr){\rm d}s
	\end{eqnarray}
Employing \eqref{eq:ink} and \eqref{eq:inkb} in \eqref{sum}, we have
		\begin{eqnarray*}
					E\big(|x(t)-x_m(t)|^2\big)&\le& C^2h^2\\\nonumber
					&+& 6\bigl(\sqrt{2}L_1h\zeta_1\bigr)^2\\\nonumber
					&+&6(\rho_1\eta_1+\sqrt{2}L_1h\eta_1)^2\int_{0}^{t}E\bigl(|x(s)-x_m(s)|^2\bigr){\rm d}s\\\nonumber
					&+&6(\sqrt{2}L_2h\zeta_2)^2\\\nonumber
					&+&6(\rho_2\eta_2+\sqrt{2}L_2h\eta_2)^2\int_{0}^{t}E\bigl(|x(s)-x_m(s)|^2\bigr){\rm d}s\\
		\end{eqnarray*}
	which implies that
		\begin{eqnarray}
		E\big(|x(t)-x_m(t)|^2\big)&\le&\bigl(C^2h^2+6\bigl(\sqrt{2}L_1h\zeta_1\bigr)^2+6(\sqrt{2}L_2h\zeta_2)^2\bigr)\\\nonumber
		&+&\bigl(6(\rho_1\eta_1+\sqrt{2}L_1h\eta_1)^2+6(\rho_2\eta_2+\sqrt{2}L_2h\eta_2)^2\bigr)\int_{0}^{t}E\bigl(|x(s)-x_m(s)|^2\bigr){\rm d}s.
	\end{eqnarray}
Let \begin{eqnarray*}
	R_1&=&\bigl(C^2h^2+6\bigl(\sqrt{2}L_1h\zeta_1\bigr)^2+6(\sqrt{2}L_2h\zeta_2)^2\bigr)\\
	R_2&=&\bigl(6(\rho_1\eta_1+\sqrt{2}L_1h\eta_1)^2+6(\rho_2\eta_2+\sqrt{2}L_2h\eta_2)^2\bigr)
\end{eqnarray*}
By using Gronwall's inequality, we have
\begin{eqnarray}
	E\big(|x(t)-x_m(t)|^2\big)&\le&R_1\exp\biggl(\int_{0}^{t}R_2{\rm d}s\biggr).
\end{eqnarray}
which implies that,
\begin{equation}
\|x(t)-x_m(t)\|_2^2	=E\big(|x(t)-x_m(t)|^2\big)\le R_1e^{R_2}=O(h^2)
\end{equation} 
		\end{proof}
\section{Numerical Examples} \label{sect5}
\noindent This section employs the proposed method to solve the NLSVIE. Two numerical examples are given to demonstrate the convergence of the method by comparing the approximate and analytical results. To measure the error between the two solutions, the infinity norm of the error $E$ is defined as $\lVert E \rVert_\infty = \underset{1 \leq i \leq m}{\max} \lvert X_i - Y_i \rvert$, where $X_i$ and $Y_i$ are the Walsh coefficients of the exact and approximate solutions, respectively. The number of iterations for each example is denoted as $n$, while the mean and standard deviation of the error $E$ are represented as $\bar{x_E}$ and $s_E$, respectively. All computations are performed using Matlab 2013(a).
\begin{example} \normalfont \cite{Kloeden}\label{Ex1}
	Let us consider the NLSVIE
	 \begin{equation*}
		x(t)=x_0-\frac{a^2}{2}\int_{0}^{t}\tanh(x(s))\sech^2(x(s)){\rm d}s+a\int_{0}^{t}\sech(x(s)){\rm d}B(s)
	\end{equation*}
	where, $x_0=0$ and $a=\frac{1}{30}$ and $x(t)$ represents the unknown stochastic process based on the same probability space $(\Omega,F,P)$, and $B(t)$ is a Brownian motion process. The exact solution is given by $x(t)=\sinh^{-1}(aB(t)+\sinh(x_0))$. 
\end{example}
 Table \ref{Tab1}  reports $\bar{x}_E$ and $s_E$  errors  as well as interval of confidence for mean error  with $m=16$ and $50$ iterations.  Figure  \ref{fig1}  displays the numerical and exact solutions with $m=32$ and $m=64$.   Figure  \ref{fig2}   shows  the behavior of error solutions with $m=16$,  $m=32$, and $n=50$.    Table \ref{Tab2} represents $\bar{x}_E$ and $s_E$  errors  as well as interval of confidence for mean error  with $m=32$ and $50$ iterations.
\begin{table}[H]
	\caption{ $\bar{x}_E$ and $s_E$  errors  as well as interval of confidence for mean error in Example \ref{Ex1} with $m=16$ and $50$ iterations.}
	\centering
	\begin{tabular}{l c c rr}
		\hline
		t & $\bar{x}_E$ &$s_E$ &
		\multicolumn{2}{c}{\underline{95\% interval of confidence for error mean }}\\[1ex]
		& & &Lower &Upper\\
		\hline
		0.1 &$0.0073\times10^{-5}$ &$0.0137\times10^{-5}$ &$0.0348\times10^{-6}$ &$0.0111\times10^{-5}$\\
		0.3 &$0.0330\times10^{-5}$ &$0.1151\times10^{-5}$ &$0.0109\times10^{-6}$ &$0.0649\times10^{-5}$\\
		0.5 &$0.0900\times10^{-5}$ &$0.3352\times10^{-5}$ &$0.0290\times10^{-6}$ &$0.1829\times10^{-5}$\\
		0.7 &$0.1402\times10^{-5}$ &$0.4698\times10^{-5}$ &$0.0995\times10^{-6}$ &$0.2704\times10^{-5}$\\
		0.9 &$0.1939\times10^{-5}$ &$0.6806\times10^{-5}$ &$0.0519\times10^{-6}$ &$0.3825\times10^{-5}$\\
		\hline
		
	\end{tabular}
\label{Tab1}
\end{table}
\begin{figure}[H]
	\includegraphics[width=0.99\textwidth]{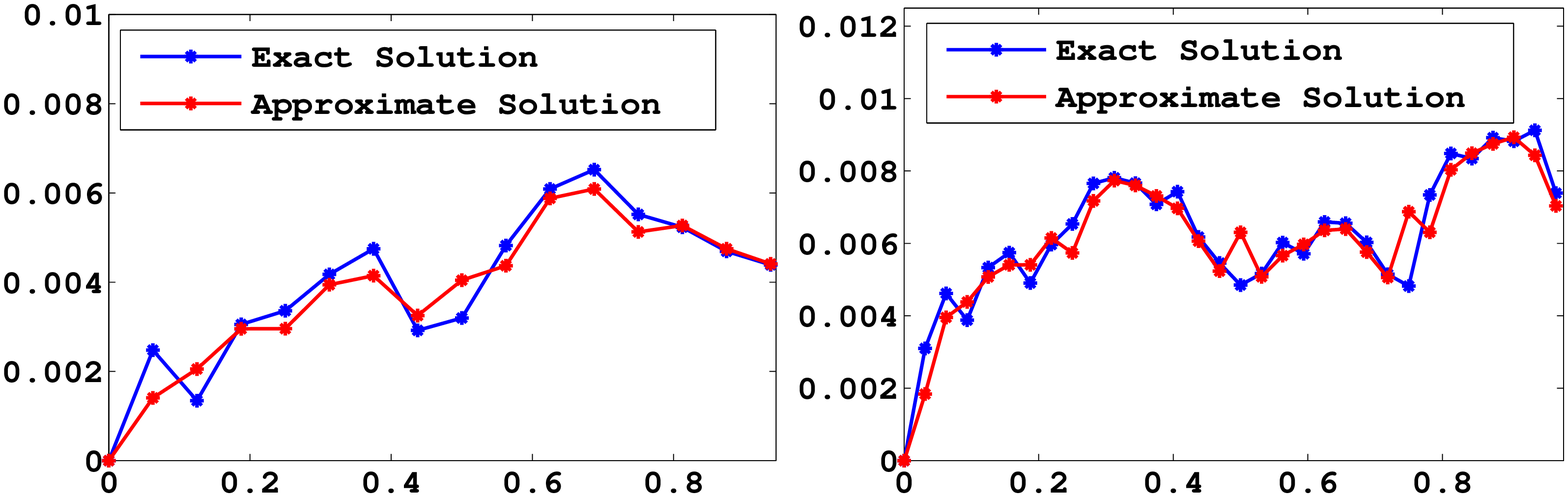}
	\caption{The numerical and exact solutions of Example \ref{Ex1} with $m=32$ and $m=64$. }
	\label{fig1}
\end{figure}
\begin{figure}[H]
	\centering
	\includegraphics[width=1\textwidth]{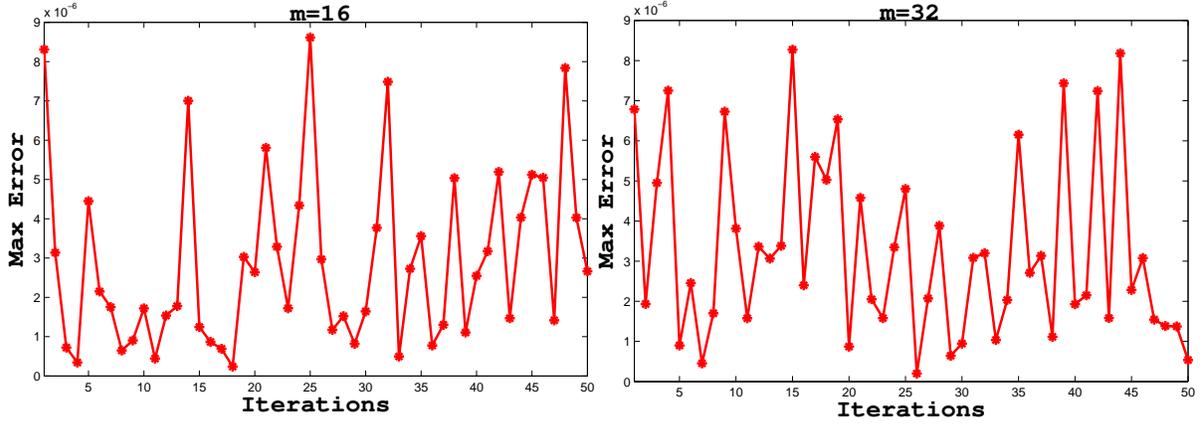}
	\caption{The behavior of error solutions in Example \ref{Ex1} with $m=16$,$m=32$, and $n=50$}
\label{fig2} \end{figure}

\begin{table}[H]
	\caption{ $\bar{x}_E$ and $s_E$  errors  as well as interval of confidence for mean error in Example \ref{Ex1} with $m=32$ and $50$ iterations.}
	\centering
	\begin{tabular}{l c c rr}
		\hline
		t & $\bar{x}_E$ &$s_E$ &
		\multicolumn{2}{c}{\underline{95\% interval of confidence for error mean }}\\[1ex]
		& & &Lower &Upper\\
		\hline
		0.1 &$0.0192\times10^{-5}$ &$0.0223\times10^{-5}$ &$0.0130\times10^{-5}$ &$0.0254\times10^{-5}$\\
		0.3 &$0.0809\times10^{-5}$ &$0.1477\times10^{-5}$ &$0.0400\times10^{-5}$ &$0.1219\times10^{-5}$\\
		0.5 &$0.1577\times10^{-5}$ &$0.3446\times10^{-5}$ &$0.0621\times10^{-5}$ &$0.2532\times10^{-5}$\\
		0.7 &$0.2262\times10^{-5}$ &$0.4336\times10^{-5}$ &$0.1061\times10^{-5}$ &$0.3464\times10^{-5}$\\
		0.9 &$0.2992\times10^{-5}$ &$0.5677\times10^{-5}$ &$0.1418\times10^{-5}$ &$0.4565\times10^{-5}$\\
		\hline
		
	\end{tabular}\label{Tab2}
\end{table}

\begin{example}\normalfont \cite{Kloeden}\label{Ex2}
	Consider the NLSVIE
	\begin{equation}\label{NLSIE}
		x(t)=x_0-a^2\int_{0}^{t}x(s)(1-x^2(s)){\rm d}s+a\int_{0}^{t}(1-x^2(s)){\rm d}B(s)
	\end{equation}
	where, $x_0=\frac{1}{10}$ and $a=\frac{1}{30}$ and $x(t)$ represents the unknown stochastic process based on the same probability space $(\Omega,F,P)$, and $B(t)$ is a Brownian motion process. The exact solution is given by $x(t)=\tanh(aB(t)+\tanh^{-1}(x_0))$.
\end{example}
Table \ref{Tab3} displays $\bar{x}_E$ and $s_E$  errors  as well as interval of confidence for mean error  with $m=16$ and $50$ iterations.   Table \ref{Tab4} presents $\bar{x}_E$ and $s_E$  errors  as well as interval of confidence for mean error  with $m=32$ and $50$ iterations.
 Figure \ref{fig3} shows  the numerical and exact solutions with $m=16$ and $m=632$.  
\begin{table}[H]
	\caption{ $\bar{x}_E$ and $s_E$  error  as well as interval of confidence for mean error in Example \ref{Ex2} with $m=16$ and $50$ iterations.}
	\centering
	\begin{tabular}{l c c rr}
		\hline
		t & $\bar{x}_E$ &$s_E$ &
		\multicolumn{2}{c}{\underline{95\% interval of confidence for error mean. }}\\[1ex]
		& & &Lower &Upper\\
		\hline
		0.1 &$0.1029\times10^{-4}$ &$0.0045\times10^{-4}$ &$0.1016\times10^{-4}$ &$0.0104\times10^{-3}$\\
		0.3 &$0.3070\times10^{-4}$ &$0.0284\times10^{-4}$ &$0.2992\times10^{-4}$ &$0.0315\times10^{-3}$\\
		0.5 &$0.5373\times10^{-4}$ &$0.0761\times10^{-4}$ &$0.5163\times10^{-4}$ &$0.0558\times10^{-3}$\\
		0.7 &$0.7689\times10^{-4}$ &$0.1193\times10^{-4}$ &$0.7358\times10^{-4}$ &$0.0802\times10^{-3}$\\
		0.9 &$0.9666\times10^{-4}$ &$0.1706\times10^{-4}$ &$0.9193\times10^{-4}$ &$0.1014\times10^{-3}$\\	
		\hline
	\end{tabular}\label{Tab3}
\end{table}
\begin{table}[H]
	\caption{ $\bar{x}_E$ and $s_E$  errors  as well as interval of confidence for mean error in Example \ref{Ex2} with $m=32$ and $50$ iterations.}
	\centering
	\begin{tabular}{l c c rr}
		\hline
		t & $\bar{x}_E$ &$s_E$ &
		\multicolumn{2}{c}{\underline{95\% interval of confidence for error mean}}\\[1ex]
		& & &Lower &Upper\\
		\hline
		0.1 &$0.1184\times10^{-4}$ &$0.0064\times10^{-4}$ &$0.1156\times10^{-4}$ &$0.0121\times10^{-3}$\\
		0.3 &$0.3161\times10^{-4}$ &$0.0281\times10^{-4}$ &$0.3038\times10^{-4}$ &$0.0328\times10^{-3}$\\
		0.5 &$0.5215\times10^{-4}$ &$0.0832\times10^{-4}$ &$0.4850\times10^{-4}$ &$0.0557\times10^{-3}$\\
		0.7 &$0.7284\times10^{-4}$ &$0.1288\times10^{-4}$ &$0.6719\times10^{-4}$ &$0.0784\times10^{-3}$\\
		0.9 &$0.9216\times10^{-4}$ &$0.1893\times10^{-4}$ &$0.8387\times10^{-4}$ &$0.1004\times10^{-3}$\\	
		\hline
		
	\end{tabular}\label{Tab4}
\end{table}
\begin{figure}[H]
	\hspace*{1cm}\includegraphics[width=1\textwidth]{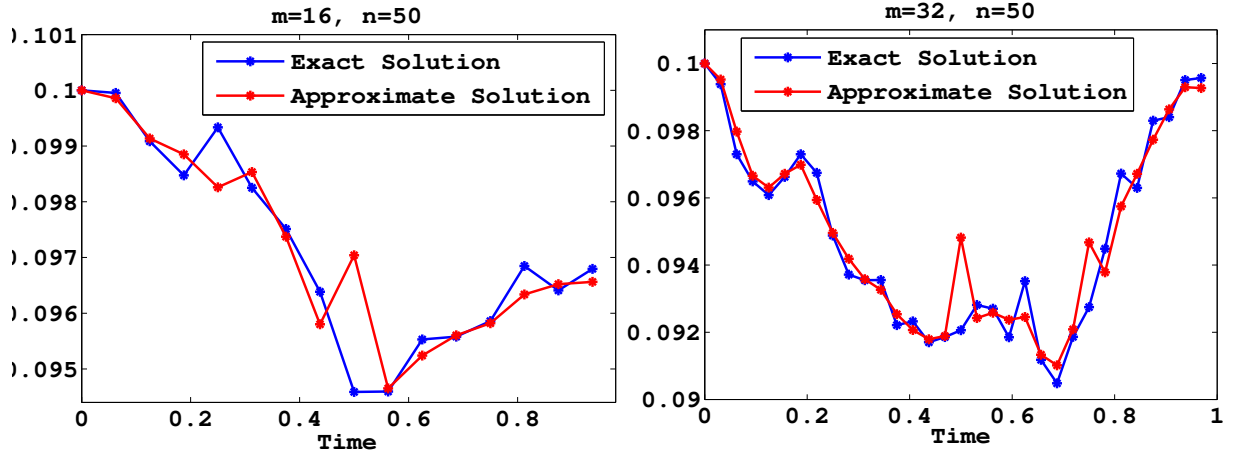}
	\caption{The numerical and exact solutions of Example \ref{Ex2} with $m=16$ and $m=32$. }
	\label{fig3}
\end{figure}

\section{Conclusion} \label{sect6}
\noindent This paper implemented the proposed numerical approach based on the Walsh function to solve the integral equation, which transforms the problem into a set of algebraic equations. These equations are then solved to obtain an approximation of the solution. Through convergence and error analysis, the method's order of convergence is found to be $\mathcal{O}(h)$. The method was applied to numerical examples with known exact solutions in the final section, and the results are presented in tables and figures. These results demonstrated that the method is highly accurate and efficient, with high precision achievable using a limited number of basis functions. Increasing the number of basis functions reduces absolute errors. Additionally, the method can be adapted to solve non-linear stochastic integral equations with fractional Brownian motion.

 \end{document}